\documentclass[reqno,A4paper]{amsart}

\usepackage{amsmath}
\usepackage{amssymb}
\usepackage{amsthm}
\usepackage{enumerate}
\usepackage{mathrsfs}

\usepackage{array}

\theoremstyle{plain}
\newtheorem{theorem}{Theorem}[section]

\newtheorem{lemma}{Lemma}[section]

\theoremstyle{definition}
\newtheorem{definition}{Definition}

\newtheorem{example}{\textit{Example}}

\numberwithin{equation}{section}

\begin{document}

\title[Stancu type operators including generalized Brenke polynomials]%
{Stancu type operators including generalized Brenke polynomials}
\author[Sezgin Sucu]%
{Sezgin Sucu}

\newcommand{\acr}{\newline\indent}

\address{Ankara University\acr
                   Faculty of Science\acr
                   Department of Mathematics\acr
                   TR-06100 Ankara\acr
                   Turkey}
\email{ssucu@ankara.edu.tr}


\begin{abstract}
Our objective in this paper is to present the sequence of Stancu type
operators including generalized Brenke polynomials. We answer the problem of
uniform approximation of continuous functions on closed bounded interval and
the problem of the order of this convergence estimate by known tools in
approximation theory. Furthermore, we apply some of the results obtained in
this work to Miller-Lee polynomials and Gould-Hopper polynomials.
\end{abstract}

\maketitle

\section{Introduction}

Approximation of functions by polynomials is not only an important topic of
the theory of mathematical analysis but also provides powerful mathematical
tools to application areas. Theorem of Weierstrass declares that every
function $f$ which is continuous in a finite closed interval $\left[ a,b%
\right] $ can be developed according to polynomials in a series which is
uniformly convergent on $\left[ a,b\right] $. According to Weierstrass'
theorem, the family of all polynomials is dense in $C\left[ a,b\right] $.
Taking advantage of a probabilistic construction, Bernstein presented a
definitively proof of the Weierstrass approximation theory.

For $\phi \in C\left[ 0,\infty \right) $ one of the most important sequence
of operators defined by Szasz \cite{1}%
\begin{equation}
S_{n}\left( \phi ;x\right) :=e^{-nx}\sum\limits_{k=0}^{\infty }\frac{\left(
nx\right) ^{k}}{k!}\phi \left( \frac{k}{n}\right) \text{,}  \label{1}
\end{equation}%
where the above sum converges under suitable conditions.

Assuming that%
\begin{equation*}
h\left( z\right) =\sum\limits_{k=0}^{\infty }a_{k}z^{k}\text{, \ \ \ }\left(
a_{0}\neq 0\right)
\end{equation*}%
is an analytic function on the following set
\begin{equation*}
\left\{ z:\left\vert z\right\vert <R,\text{ }R>1\right\}
\end{equation*}
and $h\left( 1\right) \neq 0$. If polynomials $\pi _{k}$ satisfy the
following relation
\begin{equation}
h\left( u\right) e^{ux}=\sum\limits_{k=0}^{\infty }\pi _{k}\left( x\right)
u^{k}\text{,}  \label{2}
\end{equation}%
then it is called Appell polynomials \cite{2}. These type of polynomials
mentioned above are among the most important special functions and have very
various applications to engineering and mathematical analysis.

Jakimovski and Leviatan \cite{3} gave sequence of operators $\left\{
J_{n}\right\} _{n\geq 1}$ as follows
\begin{equation}
J_{n}\left( \phi ;x\right) :=\frac{e^{-nx}}{h\left( 1\right) }%
\sum\limits_{k=0}^{\infty }\pi _{k}\left( nx\right) \phi \left( \frac{k}{n}%
\right) \text{,}  \label{3}
\end{equation}%
where $\pi _{k}\left( x\right) \geq 0$ for $x\in \left[ 0,\infty \right) $.
The proof of convergence follows in the same way as the proof of Szasz. This
type of method plays an important role in the problem of approximation
theory.

Much of the developments of approximation of functions by the sequence of
operators involving special functions was performed by many mathematicians (%
\cite{4}-\cite{14}).

The construction of our operators is based on the following relation%
\begin{equation}
A_{1}\left( h\left( t\right) \right) A_{2}\left( xh\left( t\right) \right)
=\sum\limits_{k=0}^{\infty }\pi _{k}\left( x\right) t^{k}\text{,}  \label{4}
\end{equation}%
where $A_{1}$, $A_{2}$ and $h$ are analytic functions on the following set
\begin{equation*}
\left\{ t:\left\vert t\right\vert <R,\text{ }R>1\right\}
\end{equation*}%
such that%
\begin{eqnarray}
A_{1}\left( t\right) &=&\sum\limits_{k=0}^{\infty }a_{1,k}t^{k}\ \left(
a_{1,0}\neq 0\right) \text{,}  \notag \\
A_{2}\left( t\right) &=&\sum\limits_{k=0}^{\infty }a_{2,k}t^{k}\ \left(
a_{2,k}\neq 0\right) \text{,}\ \ h\left( t\right) =\sum\limits_{k=1}^{\infty
}h_{k}t^{k}\ \left( h_{1}\neq 0\right) \text{.}  \label{5}
\end{eqnarray}%
If polynomials $\pi _{k}$ satisfy the relation $\left( \ref{4}\right) $,
then it is called generalized Brenke polynomials \cite{15}.

We are mainly concerned with the study of Varma et al. \cite{16}. Therefore,
for $\nu _{1},$ $\nu _{2}\geq 0$ we define the sequence of operators $%
\left\{ \mathcal{L}_{n}^{\left( \nu _{1},\nu _{2}\right) }\right\} _{n\geq 1}
$as follows%
\begin{equation}
\mathcal{L}_{n}^{\left( \nu _{1},\nu _{2}\right) }\left( f;x\right) =\frac{1%
}{A_{1}\left( h\left( 1\right) \right) A_{2}\left( nxh\left( 1\right)
\right) }\sum\limits_{k=0}^{\infty }\pi _{k}\left( nx\right) f\left( \frac{%
k+\nu _{1}}{n+\nu _{2}}\right) \text{,}  \label{6}
\end{equation}%
where $\pi _{k}$ polynomials are defined by $\left( \ref{4}\right) $, the
function $h$ and the polynomials $\pi _{k}$ have the following properties $%
h^{^{\prime }}\left( 1\right) =1$ and $\pi _{k}\left( x\right) \geq 0$. In
the all section of our study, we assume that
\begin{equation}
\begin{array}{ll}
\left( i\right) & A_{1},\text{ }A_{2}:%
\mathbb{R}
\longrightarrow \left( 0,\infty \right) \text{,} \\
\left( ii\right) & \left( \ref{4}\right) \text{ and }\left( \ref{5}\right)
\text{ converge for }\left\vert t\right\vert <R\ \left( R>1\right) \text{.}%
\end{array}
\label{7}
\end{equation}

In this paper we are interested in defining the sequence of Stancu type
operators including generalized Brenke polynomials. In the second section,
we use the theorem of Korovkin on the sequence of positive linear operators
and answer the problem of the order of this convergence estimate by known
tools in approximation theory. Furthermore, we apply some of the results
obtained in this work to Miller-Lee polynomials and Gould-Hopper polynomials.

\section{Approximation to functions using $\mathcal{L}_{n}^{\left( \protect%
\nu _{1},\protect\nu _{2}\right) }$ operators}

Let us denote set of functions of form%
\begin{equation*}
\frac{f\left( x\right) }{1+x^{2}}\text{ is convergent as }x\rightarrow \infty
\end{equation*}%
by $E$.

The following lemmas give rise to important observations.

\begin{lemma}
If $\pi _{k}$ are polynomials satisfying the relation $\left( \ref{4}\right)
$, then
\begin{eqnarray*}
\sum\limits_{k=0}^{\infty }\pi _{k}\left( nx\right) &=&A_{1}\left( h\left(
1\right) \right) A_{2}\left( nxh\left( 1\right) \right) \text{,} \\
\sum\limits_{k=0}^{\infty }k\pi _{k}\left( nx\right) &=&A_{1}^{^{\prime
}}\left( h\left( 1\right) \right) A_{2}\left( nxh\left( 1\right) \right)
+nxA_{1}\left( h\left( 1\right) \right) A_{2}^{^{\prime }}\left( nxh\left(
1\right) \right) \text{,} \\
\sum\limits_{k=0}^{\infty }k^{2}\pi _{k}\left( nx\right) &=&\left[ \left(
h^{^{\prime \prime }}\left( 1\right) +1\right) A_{1}^{^{\prime }}\left(
h\left( 1\right) \right) +A_{1}^{^{\prime \prime }}\left( h\left( 1\right)
\right) \right] A_{2}\left( nxh\left( 1\right) \right) \\
&&+\left[ 2A_{1}^{^{\prime }}\left( h\left( 1\right) \right) +\left(
h^{^{\prime \prime }}\left( 1\right) +1\right) A_{1}\left( h\left( 1\right)
\right) \right] A_{2}^{^{\prime }}\left( nxh\left( 1\right) \right) nx \\
&&+A_{1}\left( h\left( 1\right) \right) A_{2}^{^{\prime \prime }}\left(
nxh\left( 1\right) \right) \left( nx\right) ^{2}\text{.}
\end{eqnarray*}
\end{lemma}

As a direct consequence of Lemma 2.1, we have the following:

\begin{lemma}
For $n\geq 1$, one has the following identities%
\begin{eqnarray*}
\mathcal{L}_{n}^{\left( \nu _{1},\nu _{2}\right) }\left( 1;x\right) &=&1%
\text{,} \\
\mathcal{L}_{n}^{\left( \nu _{1},\nu _{2}\right) }\left( s;x\right) &=&\frac{%
A_{2}^{^{\prime }}\left( nxh\left( 1\right) \right) n}{A_{2}\left( nxh\left(
1\right) \right) \left( n+\nu _{2}\right) }x+\left( \frac{A_{1}^{^{\prime
}}\left( h\left( 1\right) \right) }{A_{1}\left( h\left( 1\right) \right) }%
+\nu _{1}\right) \frac{1}{n+\nu _{2}}\text{,} \\
\mathcal{L}_{n}^{\left( \nu _{1},\nu _{2}\right) }\left( s^{2};x\right) &=&%
\frac{A_{2}^{^{\prime \prime }}\left( nxh\left( 1\right) \right) n^{2}}{%
A_{2}\left( nxh\left( 1\right) \right) \left( n+\nu _{2}\right) ^{2}}x^{2} \\
&&+\frac{\left[ \left( 1+2\nu _{1}+h^{^{\prime \prime }}\left( 1\right)
\right) A_{1}\left( h\left( 1\right) \right) +2A_{1}^{^{\prime }}\left(
h\left( 1\right) \right) \right] A_{2}^{^{\prime }}\left( nxh\left( 1\right)
\right) n}{A_{1}\left( h\left( 1\right) \right) A_{2}\left( nxh\left(
1\right) \right) \left( n+\nu _{2}\right) ^{2}}x \\
&&+\frac{\nu _{1}^{2}A_{1}\left( h\left( 1\right) \right) +\left( 1+2\nu
_{1}+h^{^{\prime \prime }}\left( 1\right) \right) A_{1}^{^{\prime }}\left(
h\left( 1\right) \right) +A_{1}^{^{\prime \prime }}\left( h\left( 1\right)
\right) }{A_{1}\left( h\left( 1\right) \right) \left( n+\nu _{2}\right) ^{2}}%
\text{.}
\end{eqnarray*}

\begin{proof}
The validity of these identities follows from the definition of operators $%
\mathcal{L}_{n}^{\left( \nu _{1},\nu _{2}\right) }$ and Lemma 2.1. Together, $%
\left( \ref{4}\right) $ and $\left( \ref{6}\right) $ imply
\begin{eqnarray*}
\mathcal{L}_{n}^{\left( \nu _{1},\nu _{2}\right) }\left( 1;x\right) &=&\frac{%
1}{A_{1}\left( h\left( 1\right) \right) A_{2}\left( nxh\left( 1\right)
\right) }\sum\limits_{k=0}^{\infty }\pi _{k}\left( nx\right) \\
&=&1\text{.}
\end{eqnarray*}%
Comparing $\left( \ref{6}\right) $ and Lemma 2.1, we see
\begin{eqnarray*}
\mathcal{L}_{n}^{\left( \nu _{1},\nu _{2}\right) }\left( s;x\right) &=&\frac{%
1}{A_{1}\left( h\left( 1\right) \right) A_{2}\left( nxh\left( 1\right)
\right) }\sum\limits_{k=0}^{\infty }\pi _{k}\left( nx\right) \frac{k+\nu _{1}%
}{n+\nu _{2}} \\
&=&\frac{1}{A_{1}\left( h\left( 1\right) \right) A_{2}\left( nxh\left(
1\right) \right) \left( n+\nu _{2}\right) }\sum\limits_{k=0}^{\infty }\pi
_{k}\left( nx\right) k \\
&&+\frac{\nu _{1}}{A_{1}\left( h\left( 1\right) \right) A_{2}\left(
nxh\left( 1\right) \right) \left( n+\nu _{2}\right) }\sum\limits_{k=0}^{%
\infty }\pi _{k}\left( nx\right) \\
&=&\frac{A_{2}^{^{\prime }}\left( nxh\left( 1\right) \right) n}{A_{2}\left(
nxh\left( 1\right) \right) \left( n+\nu _{2}\right) }x+\left( \frac{%
A_{1}^{^{\prime }}\left( h\left( 1\right) \right) }{A_{1}\left( h\left(
1\right) \right) }+\nu _{1}\right) \frac{1}{n+\nu _{2}}\text{.}
\end{eqnarray*}%
The following remaining relation
\begin{eqnarray*}
\mathcal{L}_{n}^{\left( \nu _{1},\nu _{2}\right) }\left( s^{2};x\right) &=&%
\frac{1}{A_{1}\left( h\left( 1\right) \right) A_{2}\left( nxh\left( 1\right)
\right) }\sum\limits_{k=0}^{\infty }\pi _{k}\left( nx\right) \left( \frac{%
k+\nu _{1}}{n+\nu _{2}}\right) ^{2} \\
&=&\frac{1}{A_{1}\left( h\left( 1\right) \right) A_{2}\left( nxh\left(
1\right) \right) \left( n+\nu _{2}\right) ^{2}} \\
&&\times \sum\limits_{k=0}^{\infty }\pi _{k}\left( nx\right) \left(
k^{2}+2k\nu _{1}+\nu _{1}^{2}\right) \\
&=&\frac{A_{2}^{^{\prime \prime }}\left( nxh\left( 1\right) \right) n^{2}}{%
A_{2}\left( nxh\left( 1\right) \right) \left( n+\nu _{2}\right) ^{2}}x^{2} \\
&&+\frac{\left[ \left( 1+2\nu _{1}+h^{^{\prime \prime }}\left( 1\right)
\right) A_{1}\left( h\left( 1\right) \right) +2A_{1}^{^{\prime }}\left(
h\left( 1\right) \right) \right] A_{2}^{^{\prime }}\left( nxh\left( 1\right)
\right) n}{A_{1}\left( h\left( 1\right) \right) A_{2}\left( nxh\left(
1\right) \right) \left( n+\nu _{2}\right) ^{2}}x \\
&&+\frac{\nu _{1}^{2}A_{1}\left( h\left( 1\right) \right) +\left( 1+2\nu
_{1}+h^{^{\prime \prime }}\left( 1\right) \right) A_{1}^{^{\prime }}\left(
h\left( 1\right) \right) +A_{1}^{^{\prime \prime }}\left( h\left( 1\right)
\right) }{A_{1}\left( h\left( 1\right) \right) \left( n+\nu _{2}\right) ^{2}}
\end{eqnarray*}%
can be obtained from $\left( \ref{6}\right) $ and Lemma 2.1.
\end{proof}
\end{lemma}

\begin{lemma}
Let $\mathcal{L}_{n}^{\left( \nu _{1},\nu _{2}\right) }$ be operators
defined by $\left( \ref{6}\right) $. Then it follows that%
\begin{eqnarray*}
\mathcal{L}_{n}^{\left( \nu _{1},\nu _{2}\right) }\left( s-x;x\right)
&=&\left( \frac{A_{2}^{^{\prime }}\left( nxh\left( 1\right) \right) n}{%
A_{2}\left( nxh\left( 1\right) \right) \left( n+\nu _{2}\right) }-1\right) x
\\
&&+\left( \frac{A_{1}^{^{\prime }}\left( h\left( 1\right) \right) }{%
A_{1}\left( h\left( 1\right) \right) }+\nu _{1}\right) \frac{1}{n+\nu _{2}}%
\text{,} \\
\mathcal{L}_{n}^{\left( \nu _{1},\nu _{2}\right) }\left( \left( s-x\right)
^{2};x\right) &=&\left[ \frac{A_{2}^{^{\prime \prime }}\left( nxh\left(
1\right) \right) n^{2}}{A_{2}\left( nxh\left( 1\right) \right) \left( n+\nu
_{2}\right) ^{2}}-\frac{2A_{2}^{^{\prime }}\left( nxh\left( 1\right) \right)
n}{A_{2}\left( nxh\left( 1\right) \right) \left( n+\nu _{2}\right) }+1\right]
x^{2} \\
&&+\left\{ \frac{\left( 1+2\nu _{1}+h^{^{\prime \prime }}\left( 1\right)
\right) A_{2}^{^{\prime }}\left( nxh\left( 1\right) \right) n}{A_{2}\left(
nxh\left( 1\right) \right) \left( n+\nu _{2}\right) ^{2}}\right. \\
&&\left. +\frac{2A_{1}^{^{\prime }}\left( h\left( 1\right) \right)
A_{2}^{^{\prime }}\left( nxh\left( 1\right) \right) }{A_{1}\left( h\left(
1\right) \right) A_{2}\left( nxh\left( 1\right) \right) \left( n+\nu
_{2}\right) ^{2}}n\right. \\
&&\left. -\left( \frac{A_{1}^{^{\prime }}\left( h\left( 1\right) \right) }{%
A_{1}\left( h\left( 1\right) \right) }+\nu _{1}\right) \frac{2}{n+\nu _{2}}%
\right\} x \\
&&+\frac{\nu _{1}^{2}}{\left( n+\nu _{2}\right) ^{2}} \\
&&+\frac{\left( 1+2\nu _{1}+h^{^{\prime \prime }}\left( 1\right) \right)
A_{1}^{^{\prime }}\left( h\left( 1\right) \right) +A_{1}^{^{\prime \prime
}}\left( h\left( 1\right) \right) }{A_{1}\left( h\left( 1\right) \right)
\left( n+\nu _{2}\right) ^{2}}\text{.}
\end{eqnarray*}
\end{lemma}

\begin{proof}
According to the rule of linearity of $\mathcal{L}_{n}^{\left( \nu _{1},\nu
_{2}\right) }$ operators and applying Lemma 2.2, this establishes the result.
\end{proof}

Now we define%
\begin{eqnarray*}
\Delta _{1,n}^{\left( \nu _{1},\nu _{2}\right) }\left( x\right) &:&=\mathcal{%
L}_{n}^{\left( \nu _{1},\nu _{2}\right) }\left( s-x;x\right) \text{,} \\
\Delta _{2,n}^{\left( \nu _{1},\nu _{2}\right) }\left( x\right) &:&=\mathcal{%
L}_{n}^{\left( \nu _{1},\nu _{2}\right) }\left( \left( s-x\right)
^{2};x\right) \text{.}
\end{eqnarray*}

For $i=0,1,2$, let us determine under what conditions the sequence of
operators
\begin{equation*}
\mathcal{L}_{n}^{\left( \nu _{1},\nu _{2}\right) }\left( s^{i};x\right)
\end{equation*}%
will approach $x^{i}$ respectively.

\begin{theorem}
Let $f:\left[ 0,\infty \right) \rightarrow
\mathbb{R}
$ be a continuous function belonging to class $E$. Set
\begin{equation}
\lim_{y\rightarrow \infty }\frac{A_{2}^{^{\prime }}\left( y\right) }{%
A_{2}\left( y\right) }=1\text{ and }\lim_{y\rightarrow \infty }\frac{%
A_{2}^{^{\prime \prime }}\left( y\right) }{A_{2}\left( y\right) }=1\text{.}
\label{8}
\end{equation}%
Then,
\begin{equation*}
\mathcal{L}_{n}^{\left( \nu _{1},\nu _{2}\right) }\left( f;x\right)
\rightarrow f\left( x\right)
\end{equation*}%
uniformly as $n\rightarrow \infty $ on each compact subset of $\left[
0,\infty \right) $.
\end{theorem}

\begin{proof}
Under the assumptions of the $\left( \ref{8}\right) $ and by using Lemma 2.2,
for $i=0,1,2$ we obtain the following%
\begin{equation*}
\mathcal{L}_{n}^{\left( \nu _{1},\nu _{2}\right) }\left( s^{i};x\right)
\rightarrow x^{i}
\end{equation*}%
which converge uniformly in each compact subset of $\left[ 0,\infty \right) $%
. The universal Korovkin-type property \cite{17} enables us to obtain the
assertion of theorem.
\end{proof}

Now define $\tilde{C}\left[ 0,\infty \right) $ and $C_{B}\left[ 0,\infty
\right) $\ to be the set of all uniformly continuous functions and to be the
set of all bounded and continuous functions on $\left[ 0,\infty \right) $,
respectively.

We are now in a position to obtain the order of approximation for the
functions which belongs to the space $\tilde{C}\left[ 0,\infty \right) \cap
E $.

\begin{theorem}
Let $f$ be a function belonging to the class $f\in \tilde{C}\left[ 0,\infty
\right) \cap E$, then
\begin{equation*}
\left\vert \mathcal{L}_{n}^{\left( \nu _{1},\nu _{2}\right) }\left(
f;x\right) -f\left( x\right) \right\vert \leq 2\omega \left( f;\delta
_{n}\left( x\right) \right) \text{,}
\end{equation*}%
where $\delta _{n}\left( x\right) :=\sqrt{\Delta _{2,n}^{\left( \nu _{1},\nu
_{2}\right) }\left( x\right) }$ and $\omega \left( f;.\right) $ is modulus
of continuity \cite{18} of the function $f$.
\end{theorem}

\begin{proof}
It follows from Lemma 2.2 and monotonicity property of operators $\mathcal{L}%
_{n}^{\left( \nu _{1},\nu _{2}\right) }$ that%
\begin{equation*}
\left\vert \mathcal{L}_{n}^{\left( \nu _{1},\nu _{2}\right) }\left(
f;x\right) -f\left( x\right) \right\vert \leq \mathcal{L}_{n}^{\left( \nu
_{1},\nu _{2}\right) }\left( \left\vert f\left( s\right) -f\left( x\right)
\right\vert ;x\right) \text{.}
\end{equation*}%
With the aid of the property of $\omega \left( f;.\right) $ we can write
\begin{equation*}
\left\vert \mathcal{L}_{n}^{\left( \nu _{1},\nu _{2}\right) }\left(
f;x\right) -f\left( x\right) \right\vert \leq \omega \left( f;\delta \right)
\left( 1+\frac{1}{\delta }\mathcal{L}_{n}^{\left( \nu _{1},\nu _{2}\right)
}\left( \left\vert s-x\right\vert ;x\right) \right) \text{.}
\end{equation*}%
After using Cauchy-Schwarz inequality, the above inequality leads to%
\begin{equation*}
\left\vert \mathcal{L}_{n}^{\left( \nu _{1},\nu _{2}\right) }\left(
f;x\right) -f\left( x\right) \right\vert \leq \omega \left( f;\delta \right)
\left( 1+\frac{1}{\delta }\sqrt{\Delta _{2,n}^{\left( \nu _{1},\nu
_{2}\right) }\left( x\right) }\right) \text{.}
\end{equation*}%
With $\delta :=\delta _{n}\left( x\right) =\sqrt{\Delta _{2,n}^{\left( \nu
_{1},\nu _{2}\right) }\left( x\right) }$, the proof is completed.
\end{proof}

\begin{definition}
For $\alpha $ and $M$ with $0<\alpha \leq 1$, $M>0$, a function $f$ is said
to be Lipschitz property of order $\alpha $ if it satisfies%
\begin{equation}
\left\vert f\left( t_{1}\right) -f\left( t_{2}\right) \right\vert \leq
M\left\vert t_{1}-t_{2}\right\vert ^{\alpha }\text{, \ \ \ \ \ }t_{1}\text{,}%
t_{2}\in \left[ 0,\infty \right) \text{.}  \label{9}
\end{equation}
\end{definition}

\begin{theorem}
Let $f$ be a function satisfying the condition $\left( \ref{9}\right) $.
Then for $x\geq 0$%
\begin{equation*}
\left\vert \mathcal{L}_{n}^{\left( \nu _{1},\nu _{2}\right) }\left(
f;x\right) -f\left( x\right) \right\vert \leq M\delta _{n}^{\alpha }\left(
x\right) \text{,}
\end{equation*}%
where $\delta _{n}\left( x\right) $ is defined as Theorem 2.2.
\end{theorem}

\begin{proof}
Since the $\mathcal{L}_{n}^{\left( \nu _{1},\nu _{2}\right) }$ are monotone,
then we have%
\begin{eqnarray}
\left\vert \mathcal{L}_{n}^{\left( \nu _{1},\nu _{2}\right) }\left(
f;x\right) -f\left( x\right) \right\vert &=&\left\vert \mathcal{L}%
_{n}^{\left( \nu _{1},\nu _{2}\right) }\left( f\left( s\right) -f\left(
x\right) ;x\right) \right\vert  \notag \\
&\leq &\mathcal{L}_{n}^{\left( \nu _{1},\nu _{2}\right) }\left( \left\vert
f\left( s\right) -f\left( x\right) \right\vert ;x\right)  \notag \\
&\leq &M\mathcal{L}_{n}^{\left( \nu _{1},\nu _{2}\right) }\left( \left\vert
s-x\right\vert ^{\alpha };x\right) \text{.}  \label{10}
\end{eqnarray}%
On the other hand the following inequality follows from the H\"{o}lder
inequality%
\begin{equation*}
\mathcal{L}_{n}^{\left( \nu _{1},\nu _{2}\right) }\left( \left\vert
s-x\right\vert ^{\alpha };x\right) \leq \left[ \mathcal{L}_{n}^{\left( \nu
_{1},\nu _{2}\right) }\left( 1;x\right) \right] ^{\frac{2-\alpha }{2}}\left[
\Delta _{2,n}^{\left( \nu _{1},\nu _{2}\right) }\left( x\right) \right] ^{%
\frac{\alpha }{2}}\text{.}
\end{equation*}%
From this and $\left( \ref{10}\right) $ we obtain
\begin{equation*}
\left\vert \mathcal{L}_{n}^{\left( \nu _{1},\nu _{2}\right) }\left(
f;x\right) -f\left( x\right) \right\vert \leq M\left[ \Delta _{2,n}^{\left(
\nu _{1},\nu _{2}\right) }\left( x\right) \right] ^{\alpha /2}\text{.}
\end{equation*}%
Hence the assertion of theorem holds.
\end{proof}

A norm on a linear space $C_{B}\left[ 0,\infty \right) $ is defined by
\begin{equation*}
\left\Vert f\right\Vert _{C_{B}\left[ 0,\infty \right) }=\underset{x\in %
\left[ 0,\infty \right) }{\sup }\left\vert f\left( x\right) \right\vert
\text{.}
\end{equation*}%
Furthermore, the following set of functions
\begin{equation*}
C_{B}^{2}\left[ 0,\infty \right) =\left\{ \psi \in C_{B}\left[ 0,\infty
\right) :\psi ^{^{\prime }},\text{\ }\psi ^{^{\prime \prime }}\in C_{B}\left[
0,\infty \right) \right\}
\end{equation*}%
is a normed space with
\begin{equation*}
\left\Vert \psi \right\Vert _{C_{B}^{2}\left[ 0,\infty \right) }=\left\Vert
\psi \right\Vert _{C_{B}\left[ 0,\infty \right) }+\left\Vert \psi ^{^{\prime
}}\right\Vert _{C_{B}\left[ 0,\infty \right) }+\left\Vert \psi ^{^{\prime
\prime }}\right\Vert _{C_{B}\left[ 0,\infty \right) }\text{.}
\end{equation*}

An estimate of the convergence rate is obtained in the following theorem by
using Peetre's K -functional \cite{18} which is relevant in approximation
theory.

\begin{theorem}
Suppose $f\in C_{B}\left[ 0,\infty \right) $ and $x\in \left[ 0,\infty
\right) $. Then%
\begin{equation*}
\left\vert \mathcal{L}_{n}^{\left( \nu _{1},\nu _{2}\right) }\left(
f;x\right) -f\left( x\right) \right\vert \leq 2\mathcal{K}\left( f;\lambda
_{n}\left( x\right) \right) \text{,}
\end{equation*}%
where
\begin{equation*}
\lambda _{n}\left( x\right) =\frac{1}{2}\left[ \Delta _{1,n}^{\left( \nu
_{1},\nu _{2}\right) }\left( x\right) +\Delta _{2,n}^{\left( \nu _{1},\nu
_{2}\right) }\left( x\right) \right]
\end{equation*}%
and $\mathcal{K}\left( f;.\right) $ is the Peetre's K -functional of the
function $f$.
\end{theorem}

\begin{proof}
Expand $\psi \in C_{B}^{2}\left[ 0,\infty \right) $ about $x$ using Taylor's
theorem we get
\begin{equation*}
\psi \left( s\right) =\psi \left( x\right) +\left( s-x\right) \psi
^{^{\prime }}\left( x\right) +\frac{\psi ^{^{\prime \prime }}\left( \eta
\right) }{2}\left( s-x\right) ^{2},\text{ \ \ }\eta \in \left( x,s\right)
\text{.}
\end{equation*}%
This yields%
\begin{equation*}
\mathcal{L}_{n}^{\left( \nu _{1},\nu _{2}\right) }\left( \psi ;x\right)
-\psi \left( x\right) =\psi ^{^{\prime }}\left( x\right) \Delta
_{1,n}^{\left( \nu _{1},\nu _{2}\right) }\left( x\right) +\frac{\psi
^{^{\prime \prime }}\left( \eta \right) }{2}\Delta _{2,n}^{\left( \nu
_{1},\nu _{2}\right) }\left( x\right) \text{.}
\end{equation*}%
From this it follows easily that%
\begin{eqnarray}
\left\vert \mathcal{L}_{n}^{\left( \nu _{1},\nu _{2}\right) }\left( \psi
;x\right) -\psi \left( x\right) \right\vert &\leq &\left\vert \psi
^{^{\prime }}\left( x\right) \right\vert \Delta _{1,n}^{\left( \nu _{1},\nu
_{2}\right) }\left( x\right) +\frac{\left\vert \psi ^{^{\prime \prime
}}\left( \eta \right) \right\vert }{2}\Delta _{2,n}^{\left( \nu _{1},\nu
_{2}\right) }\left( x\right)  \notag \\
&\leq &\left[ \Delta _{1,n}^{\left( \nu _{1},\nu _{2}\right) }\left(
x\right) +\Delta _{2,n}^{\left( \nu _{1},\nu _{2}\right) }\left( x\right) %
\right] \left\Vert \psi \right\Vert _{C_{B}^{2}\left[ 0,\infty \right) }%
\text{.}  \label{11}
\end{eqnarray}%
With Lemma 2.2 and expression $\left( \ref{11}\right) $ the estimation can be
written as%
\begin{eqnarray*}
\left\vert \mathcal{L}_{n}^{\left( \nu _{1},\nu _{2}\right) }\left(
f;x\right) -f\left( x\right) \right\vert &\leq &\left\vert \mathcal{L}%
_{n}^{\left( \nu _{1},\nu _{2}\right) }\left( f-\psi ;x\right) \right\vert
+\left\vert \mathcal{L}_{n}^{\left( \nu _{1},\nu _{2}\right) }\left( \psi
;x\right) -\psi \left( x\right) \right\vert \\
&&+\left\vert f\left( x\right) -\psi \left( x\right) \right\vert \\
&\leq &2\left\Vert f-\psi \right\Vert _{C_{B}\left[ 0,\infty \right)
}+\left\vert \mathcal{L}_{n}^{\left( \nu _{1},\nu _{2}\right) }\left( \psi
;x\right) -\psi \left( x\right) \right\vert \\
&\leq &2\left( \left\Vert f-\psi \right\Vert _{C_{B}\left[ 0,\infty \right)
}+\lambda _{n}\left( x\right) \left\Vert \psi \right\Vert _{C_{B}^{2}\left[
0,\infty \right) }\right) \text{.}
\end{eqnarray*}%
Taking the infimum over all $\psi \in C_{B}^{2}\left[ 0,\infty \right) $,
the last inequality together with the definition of $\mathcal{K}\left(
f;.\right) $ implies the following desired result%
\begin{equation*}
\left\vert \mathcal{L}_{n}^{\left( \nu _{1},\nu _{2}\right) }\left(
f;x\right) -f\left( x\right) \right\vert \leq 2\mathcal{K}\left( f;\lambda
_{n}\left( x\right) \right) \text{.}
\end{equation*}
\end{proof}

\begin{theorem}
Suppose $f\in C_{B}\left[ 0,\infty \right) $. Then%
\begin{equation*}
\left\vert \mathcal{L}_{n}^{\left( \nu _{1},\nu _{2}\right) }\left(
f;x\right) -f\left( x\right) \right\vert \leq C\omega _{2}\left( f;\sqrt{\mu
_{n}\left( x\right) }\right) +\omega \left( f;\Delta _{1,n}^{\left( \nu
_{1},\nu _{2}\right) }\left( x\right) \right) \text{,}
\end{equation*}%
where $C$ is a positive constant and%
\begin{equation*}
\mu _{n}\left( x\right) =\frac{1}{8}\left\{ \Delta _{2,n}^{\left( \nu
_{1},\nu _{2}\right) }\left( x\right) +\left[ \Delta _{1,n}^{\left( \nu
_{1},\nu _{2}\right) }\left( x\right) \right] ^{2}\right\} \text{,}
\end{equation*}%
and $\omega _{2}\left( f;.\right) $ is the second order modulus of
smoothness \cite{18} of function $f$.
\end{theorem}

\begin{proof}
Firstly let us consider the operators $\mathcal{F}_{n}^{\left( \nu _{1},\nu
_{2}\right) }$ given by%
\begin{equation*}
\mathcal{F}_{n}^{\left( \nu _{1},\nu _{2}\right) }\left( f;x\right) =%
\mathcal{L}_{n}^{\left( \nu _{1},\nu _{2}\right) }\left( f;x\right) -f\left(
\mathcal{L}_{n}^{\left( \nu _{1},\nu _{2}\right) }\left( s;x\right) \right)
+f\left( x\right) \text{.}
\end{equation*}%
Then one obtains from Lemma 2.2%
\begin{equation}
\mathcal{F}_{n}^{\left( \nu _{1},\nu _{2}\right) }\left( s-x;x\right) =0%
\text{.}  \label{12}
\end{equation}%
Moreover, for $\psi \in C_{B}^{2}\left[ 0,\infty \right) $ the following
equality can be obtained by the Taylor formula%
\begin{equation*}
\psi \left( s\right) =\psi \left( x\right) +\left( s-x\right) \psi
^{^{\prime }}\left( x\right) +\int\limits_{x}^{s}\left( s-u\right) \psi
^{^{\prime \prime }}\left( u\right) du\text{.}
\end{equation*}%
This equation, together with $\left( \ref{12}\right) $, leads immediately to%
\begin{eqnarray}
\left\vert \mathcal{F}_{n}^{\left( \nu _{1},\nu _{2}\right) }\left( \psi
;x\right) -\psi \left( x\right) \right\vert &=&\left\vert \mathcal{F}%
_{n}^{\left( \nu _{1},\nu _{2}\right) }\left( \int\limits_{x}^{s}\left(
s-u\right) \psi ^{^{\prime \prime }}\left( u\right) du;x\right) \right\vert
\notag \\
&\leq &\left\vert \mathcal{L}_{n}^{\left( \nu _{1},\nu _{2}\right) }\left(
\int\limits_{x}^{s}\left( s-u\right) \psi ^{^{\prime \prime }}\left(
u\right) du;x\right) \right\vert  \notag \\
&&+\left\vert \int\limits_{x}^{\mathcal{L}_{n}^{\left( \nu _{1},\nu
_{2}\right) }\left( s;x\right) }\left( \mathcal{L}_{n}^{\left( \nu _{1},\nu
_{2}\right) }\left( s;x\right) -u\right) \psi ^{^{\prime \prime }}\left(
u\right) du\right\vert  \notag \\
&\leq &\frac{1}{2}\left\{ \Delta _{2,n}^{\left( \nu _{1},\nu _{2}\right)
}\left( x\right) +\left[ \Delta _{1,n}^{\left( \nu _{1},\nu _{2}\right)
}\left( x\right) \right] ^{2}\right\} \left\Vert \psi ^{^{\prime \prime
}}\right\Vert _{C_{B}\left[ 0,\infty \right) }  \notag \\
&\leq &4\mu _{n}\left( x\right) \left\Vert \psi \right\Vert _{C_{B}^{2}\left[
0,\infty \right) }\text{.}  \label{13}
\end{eqnarray}%
Combining the definition of $\mathcal{F}_{n}^{\left( \nu _{1},\nu
_{2}\right) }$ operator, Lemma 2.2 and $\left( \ref{13}\right) $, we obtain
the estimate%
\begin{eqnarray*}
\left\vert \mathcal{L}_{n}^{\left( \nu _{1},\nu _{2}\right) }\left(
f;x\right) -f\left( x\right) \right\vert &\leq &\left\vert \mathcal{F}%
_{n}^{\left( \nu _{1},\nu _{2}\right) }\left( f-\psi ;x\right) -\left(
f-\psi \right) \left( x\right) \right\vert \\
&&+\left\vert \mathcal{F}_{n}^{\left( \nu _{1},\nu _{2}\right) }\left( \psi
;x\right) -\psi \left( x\right) \right\vert +\left\vert f\left( \mathcal{L}%
_{n}^{\left( \nu _{1},\nu _{2}\right) }\left( s;x\right) \right) -f\left(
x\right) \right\vert \\
&\leq &4\left\Vert f-\psi \right\Vert _{C_{B}\left[ 0,\infty \right) }+4\mu
_{n}\left( x\right) \left\Vert \psi \right\Vert _{C_{B}^{2}\left[ 0,\infty
\right) } \\
&&+\omega \left( f;\Delta _{1,n}^{\left( \nu _{1},\nu _{2}\right) }\left(
x\right) \right) \text{.}
\end{eqnarray*}%
If we take into account the relation between $\mathcal{K}\left( f;.\right) $
and $\omega _{2}\left( f;.\right) $, we have
\begin{eqnarray*}
\left\vert \mathcal{L}_{n}^{\left( \nu _{1},\nu _{2}\right) }\left(
f;x\right) -f\left( x\right) \right\vert &\leq &4\mathcal{K}\left( f;\mu
_{n}\left( x\right) \right) +\omega \left( f;\Delta _{1,n}^{\left( \nu
_{1},\nu _{2}\right) }\left( x\right) \right) \\
&\leq &C\omega _{2}\left( f;\sqrt{\mu _{n}\left( x\right) }\right) +\omega
\left( f;\Delta _{1,n}^{\left( \nu _{1},\nu _{2}\right) }\left( x\right)
\right) \text{.}
\end{eqnarray*}%
That is the assertion.
\end{proof}

\section{Examples}

In this section, we will clarify our analysis. To illustrate our situation,
we now consider two examples.

\begin{example}
The function%
\begin{equation*}
e^{bt^{d+1}}\exp \left( xt\right)
\end{equation*}%
is the generating function of the Gould-Hopper polynomials \cite{19}, i.e.,
the expansion
\begin{equation}
e^{bt^{d+1}}\exp \left( xt\right) =\sum\limits_{k=0}^{\infty
}g_{k}^{d+1}\left( x,b\right) \frac{t^{k}}{k!}  \label{14}
\end{equation}%
holds. The general expression for these polynomials is given as%
\begin{equation*}
g_{k}^{d+1}\left( x,b\right) =\sum\limits_{s=0}^{\left[ \frac{k}{d+1}\right]
}\frac{k!}{s!\left( k-\left( d+1\right) s\right) !}b^{s}x^{k-\left(
d+1\right) s}\text{.}
\end{equation*}%
The polynomials $g_{k}^{d+1}$ defined by $\left( \ref{14}\right) $ are the
generalized Brenke polynomials with
\begin{equation*}
A_{1}\left( t\right) =e^{bt^{d+1}}\text{, }A_{2}\left( t\right) =e^{t}\ \
\text{and \ }h\left( t\right) =t\text{.}
\end{equation*}%
Under the assumption $b\geq 0$, these polynomials satisfy the conditions of
Theorem 2.1 and assumptions given in introduction part of this work. Hence,
the explicit form of $\mathcal{L}_{n}^{\left( \nu _{1},\nu _{2}\right) }$
operators is given by%
\begin{equation*}
\mathcal{L}_{n}^{\left( \nu _{1},\nu _{2}\right) }\left( f;x\right)
=e^{-nx-b}\sum\limits_{k=0}^{\infty }\frac{g_{k}^{d+1}\left( nx,b\right) }{k!%
}f\left( \frac{k+\nu _{1}}{n+\nu _{2}}\right) \text{,}
\end{equation*}%
where $x\in \left[ 0,\infty \right) $.
\end{example}

\begin{example}
The function%
\begin{equation*}
\frac{1}{\left( 1-t\right) ^{m+1}}\exp \left( xt\right)
\end{equation*}%
is the generating function of the Miller-Lee polynomials \cite{20}, i.e.,
the expansion\
\begin{equation}
\frac{1}{\left( 1-t\right) ^{m+1}}\exp \left( xt\right)
=\sum\limits_{k=0}^{\infty }G_{k}^{\left( m\right) }\left( x\right) t^{k}
\label{15}
\end{equation}%
holds for\ $\left\vert t\right\vert <1$. The general expression for these
polynomials is given as%
\begin{equation*}
G_{k}^{\left( m\right) }\left( x\right) =\sum\limits_{r=0}^{k}\frac{\left(
m+1\right) _{r}}{r!\left( k-r\right) !}x^{k-r}\text{,}
\end{equation*}%
where $\left( .\right) _{r}$ is the Pochhammer's symbol. The polynomials $%
G_{k}^{\left( m\right) }$ defined by $\left( \ref{15}\right) $ are the
generalized Brenke polynomials with%
\begin{equation*}
A_{1}\left( t\right) =\frac{1}{\left( 1-t\right) ^{m+1}}\text{, }A_{2}\left(
t\right) =e^{t}\text{ \ and\ \ }h\left( t\right) =t\text{.}
\end{equation*}%
For $t\rightarrow \frac{t}{2}$ and $x\rightarrow 2x$, equation $\left( \ref%
{15}\right) $ takes the form%
\begin{equation*}
\frac{1}{\left( 1-\frac{t}{2}\right) ^{m+1}}\exp \left( xt\right)
=\sum\limits_{k=0}^{\infty }\frac{G_{k}^{\left( m\right) }\left( 2x\right) }{%
2^{k}}t^{k}\text{, \ \ \ \ }\left\vert t\right\vert <2\text{.}
\end{equation*}%
Hence, the explicit form of $\mathcal{L}_{n}^{\left( \nu _{1},\nu
_{2}\right) }$ operators is given by%
\begin{equation*}
\mathcal{L}_{n}^{\left( \nu _{1},\nu _{2}\right) }\left( f;x\right)
=e^{-nx}\sum\limits_{k=0}^{\infty }\frac{G_{k}^{\left( m\right) }\left(
2nx\right) }{2^{m+k+1}}f\left( \frac{k+\nu _{1}}{n+\nu _{2}}\right) \text{,}
\end{equation*}%
where $m>-1$ and $x\in \left[ 0,\infty \right) $.
\end{example}

\end{document}